\documentclass[12pt,reqno,tbtags]{amsart}

\usepackage{amsmath}
\usepackage{float}
\usepackage{graphicx}
\usepackage{amsthm}
\usepackage{amssymb}
\usepackage{framed}
\usepackage{tikz}
\usetikzlibrary{arrows}
\usepackage[margin=1in]{geometry}
\usepackage{caption,sansmath}
\DeclareCaptionFont{sansmath}{\sansmath}
\newtheorem{theorem}{Theorem}[section]
\newtheorem{lem}[theorem]{Lemma}

\newtheorem{corollary}[theorem]{Corollary}

\newtheorem{claim}[theorem]{Claim}

\newcommand{\mc}[1]{\mathcal{#1}}
\newcommand{\mf}[1]{\mathfrak{#1}}
\newcommand{\bb}[1]{\mathbb{#1}}
\newcommand{\brm}[1]{\operatorname{#1}}

\newcommand{\eps}{\varepsilon}
\newcommand{\pl}[1]{\brm{#1}}

\title{Tur\'{a}n Numbers of Extensions}
\author{Sergey Norin} 
\address{Department of Mathematics and Statistics, McGill University.}
\email{snorin@math.mcgill.ca}
\thanks{Supported by an NSERC grant 418520.}
\author{Liana Yepremyan}
\address{School of Computer Science, McGill University.} \email{ liana.yepremyan@mail.mcgill.ca} 
\date{}

\begin{document}

\begin{abstract}
The \emph{extension} of an $r$-uniform hypergraph $\mc{G}$ is obtained from it by adding for every pair of vertices of $\mc{G}$, which is not covered by an edge in $\mc{G}$, an extra edge containing this pair and $(r-2)$ new vertices. Keevash~\cite{keevash} and Sidorenko~\cite{sidorenko} have previously determined Tur\'an densities of two families of hypergraph extensions. We  determine the Tur\'an numbers for these families,  using classical stability techniques and new tools introduced in~\cite{gentriangle}.
\end{abstract}

\maketitle
\section{Introduction}

In this paper we employ and extend the methods introduced in~\cite{gentriangle} to determine Tur\'an numbers of extensions of a family of hypergraphs. We start by presenting the necessary definitions.

We study $r$-uniform hypergraphs, which we call $r$-graphs for brevity. We denote the vertex set of an $r$-graph $\mc{G}$ by $V(\mc{G})$ and the number of its vertices by $\brm{v}(\mc{G})$.
Let $\mf{F}$ be a family of $r$-graphs. An $r$-graph $\mc{G}$ is \emph{$\mf{F}$-free} if it does not contain any member of $\mathfrak{F}$ as a subgraph. The \emph{Tur\'{a}n number} $\brm{ex}(n,\mathfrak{F})$ is the maximum size of an $\mathfrak{F}$-free $r$-graph of order $n$:
\[\brm{ex}(n,\mathfrak{F}) = \max\left\{|\mc{G}| : \brm{v}(\mc{G})=n \textrm{ and } \mc{G} \textrm{ is } \mathfrak{F}-\textrm{free}\right\}.\]
When $\mathfrak{F}$ contains just one element, say $\mathfrak{F}=\{\mc{F}\}$, we write $\brm{ex}(n, \mc{F})= \brm{ex}(n,\mathfrak{F})$.  The \emph{Tur\'{a}n density} of the family of \(r\)-graphs \(\mf{F}\)  is defined as
\[\pi(\mf{F})  = \lim_{n \rightarrow \infty}\frac{\brm{ex}(n, \mf{F})}{{n \choose r}}.\]
We say that a pair of vertices $\{u,v\} \in V(\mc{G})$ is \emph{covered in $\mc{G}$} if $\{u,v\} \subseteq E$ for some $E \in \mc{G}$, and it is \emph{uncovered}, otherwise. We say that $\mc{G}$ \emph{covers pairs} if every pair of vertices is covered in $\mc{G}$.
Given an \(r\)-graph \(\mc{G}\), the \emph{extension of \(\mc{G}\)}, denoted by \(\brm{Ext}(\mc{G})\), is an \(r\)-graph defined as follows. For every uncovered pair $P$ in \(\mc{G}\) we add $r-2$ new vertices $v^P_1,v^P_2,\ldots,v^P_{r-2}$ to $V(\mc{G})$,  and add the edge $P \cup \{v^P_1,v^P_2,\ldots,v^P_{r-2} \}$ to $\mc{G}$. 

In~\cite{gentriangle} we determined the Tur\'{a}n number of the extension of the graph consisting of two \(r\)-edges, sharing \(r-1\) vertices, for \(r=5,6\) for large $n$ . (This graph is known as the \emph{generalized triangle}.) In this paper we consider two different families of extensions.

Our first main result is connected to the famous Erd\H{o}s-S\'{o}s conjecture from 1963, which asserts that if \(G\) is a simple graph of order \(n\) with average degree more than \(k-2\), then \(G\) contains every tree on $k$ vertices as a subgraph. This conjecture has been verified for several families of trees, and in early 1990's the proof of the conjecture for large enough \(k\) was announced by Ajtai, Koml\'{o}s, Simonovits and Szemer\'{e}di.  We say that a tree is  \emph{an Erd\H{o}s-S\'{o}s-tree} if it satisfies the conjecture.  Given a \(2\)-graph \(G\), define the \emph{\((r-2)\)-expansion of $G$} to be  the \(r\)-graph obtained by adding \((r-2)\) vertices to $G$ and  enlarging each edge of $G$ to contain these vertices. In ~\cite{sidorenko} Sidorenko proved the following.

\begin{theorem}[\cite{sidorenko}]\label{thm:sidorenko} For every \(r\geq 2\), there exists \(M_r\) such that if \(T\) is an Erd\H{o}s-S\'{o}s-tree on \(t\geq M_r\) vertices then \(\pi(\brm{Ext}(\mc{T})) =  r!(t+r-3)^{-r}{t+r-3\choose r}\), where  \(\mc{T}\) is the \((r-2)\)-expansion of \(T\).
\end{theorem}

Note that the quantity $(t+r-3)^{-r}{t+r-3\choose r}$ above is the Lagrangian of the complete \(r\)-graph on \((t+r-3)\) vertices.  We postpone the definition of the Lagrangian to Section~\ref{sec:notation}. Let $\mc{K}^{(r)}_{p}(n)$ denote the balanced blowup of $\mc{K}^{(r)}_{p}$ on $n$ vertices. That is $\pl{v}(\mc{K}^{(r)}_{p}(n))=n$,  there exists a partition $(P_1,P_2,\ldots,P_{p})$ of $V(\mc{K}^{(r)}_{p}(n))$ such that $|P_i| \in \{\lfloor\frac{n}{p} \rfloor, \lceil \frac{n}{p} \rceil\}$ for every $i \in [p]$, and an $r$-element subset of $V(\mc{K}^{(r)}_{p}(n))$ is an edge if and only if it contains at most one element of each $P_i$. We prove the following exact version of Theorem~\ref{thm:sidorenko}.

\begin{theorem}
\label{thm:trees} For every \(r\geq 2\), there exists \(M_r\) such that the following holds. Let \(T\) be an Erd\H{o}s-S\'{o}s-tree on \(t\geq M_r\) vertices and let $\mc{T}$ be the $(r-2)$ expansion of $T$. Then there exists \(n_0\) such that  \(\mc{K}^{(r)}_{t+r-3}(n)\) is the unique \(\brm{Ext}(\mc{T})\)-free \(r\)-graph on \(n\) vertices with the maximum number of edges for all \(n\geq n_0\). 
\end{theorem}

Our second result concerns extensions of a different class of sparse hypergraphs.
Let $\bar{\mc{K}}_t^{(r)}$ denote the edgeless $r$-graph on $t$ vertices. Mubayi~\cite{mubayi} determined $\pi(\brm{Ext}(\bar{\mc{K}}_t^{(r)}))$ and  Pikhurko~\cite{pikhurko} obtained the corresponding exact result.

\begin{theorem}[\cite{pikhurko}] For every \(t>r\geq 3\) there exists \(n_0\) such that \(\mc{K}^{(r)}_t(n)\) is the unique $\brm{Ext}(\bar{\mc{K}}_{t+1}^{(r)})$-free \(r\)-graph on \(n\) vertices with the maximum number of edges for all \(n\geq n_0\).
\end{theorem}

Keevash~\cite{keevash} considered the following generalization of the above problem. 
Let \(\mc{F}\) be any \(r\)-graph that covers pairs, and let $\mc{F}^{+t}$ be obtained from $\mc{F}$ by adding new isolated vertices so that $\pl{v}(\mc{F}^{+t})=t$. (We have $\emptyset^{+t}=\bar{\mc{K}}_{t}^{(r)}$, where $\emptyset$ denotes the null $r$-graph.) In~\cite{keevash} Keevash, generalizing the density argument from~\cite{mubayi}, proved the following.

\begin{theorem}[\cite{keevash}]\label{thm:keevash} Let \(\mc{F}\) be an \(r\)-graph that covers pairs with \(\pl{v}(\mc{F}) \leq t+1\). If \(\pi(\mc{F}) \leq r!t^{-r}{t\choose r}\), then \(\pi(\brm{Ext}(\mc{F}^{+(t+1)})) =r!t^{-r}{t\choose r}\). 
\end{theorem}

We obtain the exact version of a slight weakening of Theorem~\ref{thm:keevash}. 

\begin{theorem}\label{thm:mubayigen} Let \(\mc{F}\) be an \(r\)-graph that covers pairs with \(\pl{v}(\mc{F})\leq t\). If \(\pi(\mc{F}) <r!t^{-r}{t\choose r}\) then there exists \(n_0\) such that \(\mc{K}^{(r)}_t(n)\) is the unique \(\brm{Ext}(\mc{F}^{+(t+1)})\)-free \(r\)-graph on \(n\) vertices with maximum number of edges for all \(n\geq n_0\).
\end{theorem}
 
Our proofs of Theorems~\ref{thm:trees} and~\ref{thm:mubayigen} mainly utilize the classical stability method introduced by Erd\H{o}s and Simonovits in~\cite{simonovits}. In fact, one of  our main technical results, Theorem~\ref{vertexlocalstab} can be considered as a hypergraph analogue of the result in~\cite{simonovits}, which states that for every   $t$-critical $2$-graph $G$ there exists $n_0$ such that the Tur\'an graph $K_t(n)$ is the unique $G$-free graph on $n$ vertices for every $n \geq n_0$. We also use several of the tools introduced in~\cite{gentriangle} to streamline stability and ``symmetrization" arguments.
 
The rest of this paper is organized as follows. In Section~\ref{sec:notation} we introduce methods from~\cite{gentriangle} and outline our approach. In Section~\ref{sec:localstab} we describe a large family of hypergraphs for which we are able to prove the upper bound on the number of edges locally, that is when the graph is close to the conjectured extremal example.  Sections~\ref{sec:trees} and~\ref{sec:mubayipikh}  contain the proofs of Theorem ~\ref{thm:trees} and~\ref{thm:mubayigen}, respectively. 

\section{Notation and Preliminary Results}
\label{sec:notation}
\subsection{Notation}
We adopt most of the notation and use several results from~\cite{gentriangle}. Let
$[n]=\{1,2,\ldots,n\}$. Let $X^{(r)}$ denote the set of all $r$ element subsets of a set $X$.
For an $r$-graph $\mc{F}$ and $v \in V(\mc{F})$,  the \emph{link of the vertex \(v\)} is defined as $$ L_{\mc{F}}(v):=\{I \in (V(\mc{F}))^{(r-1)} \: | \: I \cup \{v\} \in \mc{F} \}.$$ More generally, for  \(I\subseteq V(\mc{F})\) the \emph{link $L_{\mc{F}}(I)$  of \(I\)} is defined as 
$$L_{\mc{F}}(I) := \{J\subseteq V(\mc{F}) \: | \:  J \cap I = \emptyset, I \cup J\in \mc{F}\}.$$
We skip the index $\mc{F}$, whenever $\mc{F}$ is understood from the context. 

We say that an $r$-graph $\mc{G}$ is obtained from an $r$-graph $\mc{F}$ by \emph{cloning a vertex $v$ to a set $W$} if $\mc{F} \subseteq \mc{G}$, $V(\mc{G}) \setminus V(\mc{F}) = W\setminus\{v\}$, and $ L_{\mc{G}}(w)= L_{\mc{F}}(v)$ for every $w \in W$. We say that $\mc{G}$ is \emph{a blowup of $\mc{F}$} if $\mc{G}$ is isomorphic to an $r$-graph obtained from $\mc{F}$ by repeatedly cloning and deleting vertices. We denote the set of all blowups of $\mc{F}$ by $\mf{B}(\mc{F})$.

For a family of $r$-graphs $\mf{F}$, let 
$$m(\mf{F},n):=\max_{\substack{\mc{F} \in \mf{F} \\ \brm{v}({\mc{F}}) = n}} |\mc{F}|$$
denote the maximum number of edges in an $r$-graph in  $\mf{F}$ on $n$ vertices.

\subsection{Stability}
\label{sec:stability}
Let $\mf{F}$ and $\mf{H}$ be two families of $r$-graphs. We define \emph{the distance $d_{\mf{F}}(\mc{F})$ from an $r$-graph $\mc{F}$ to a family $\mf{F}$} as 
\[ d_{\mf{F}}(\mc{F}):=\min_{\substack{\mc{F'}\in\mf{F} \\ \brm{v}(\mc{F}) = \brm{v}(\mc{F}')}}{|\mc{F}\triangle\mc{F}'|}.\]

For $\eps, \alpha>0 $, we say that $\mf{F}$ is $(\mf{H}, \eps, \alpha)$-\emph{locally stable} if there exists $n_0 \in \mathbb{N}$ such that for all $\mc{F}\in\mf{F}$ with $\brm{v}(\mc{F}) =n \geq n_0$ and $d_{\mf{H}}(\mc{F})\leq \eps n^r$ we have
\begin{equation}\label{eq:localstability}
|\mc{F}|\leq m(\mf{H},n) - \alpha d_{\mf{H}}(\mc{F}).
\end{equation}
We say that $\mf{F}$ is $\mf{H}$-\emph{locally stable}  if $\mf{F}$ is $(\mf{H}, \eps, \alpha)$-locally stable for some choice of $\eps$ and $\alpha$. We say that $\mf{F}$ is $(\mf{H}, \alpha)$-\emph{stable} if it is $(\mf{H}, 1, \alpha)$-\emph{locally stable}, that is the inequality (\ref{eq:localstability}) holds for all  $\mc{F}\in\mf{F}$ with $\brm{v}(\mc{F}) =n \geq n_0$. We say that \(\mf{F}\) is \(\mf{H}\)-stable, if \(\mf{F}\) is \((\mf{H}, \alpha)\)-stable for some choice of \(\alpha\). We refer the reader to~\cite{gentriangle} for the detailed discussion of this notion of stability and its differences from the classical definitions. 

For $\eps, \alpha>0$, we say that a family $\mf{F}$ of $r$-graphs is  $(\mf{H}, \eps, \alpha)$-\emph{vertex locally stable} if there exists $n_0 \in \mathbb{N}$ such that for all $\mc{F}\in\mf{F}$ with $\brm{v}(\mc{F}) =n \geq n_0$,  $d_{\mf{H}}(\mc{F})\leq \eps n^{r}$, and 
$| L_{\mc{F}}(v)| \geq r(1-\eps)m(\mf{H},n)/n$  for every $v \in V(\mc{F})$, we have
\[|\mc{F}|\leq m(\mf{H},n) - \alpha d_{\mf{H}}(\mc{F}).\]
We say that $\mf{F}$ is  $\mf{H}$-\emph{vertex locally stable} if $\mf{F}$ is  $(\mf{H}, \eps, \alpha)$-vertex locally stable for some $\eps,\alpha$. 
It is shown in~\cite{gentriangle} that vertex local stability implies local stability under mild conditions.

\begin{theorem}[{\cite[Theorem 3.1]{gentriangle}}]\label{localfromvertexlocal} Let \(\mf{F}, \mf{H}\) be two families of \(r\)-graphs such that \(\mf{H}\) is clonable. If \(\mf{F}\) is \(\mf{H}\)-vertex locally stable, then \(\mf{F}\) is \(\mf{H}\)-locally stable.
\end{theorem}

\subsection{Lagrangians and weighted stability}\label{sec:weighted}

For $\mc{F}$ an $r$-graph $\mc{F}$ let $\mc{M}(\mc{F})$ denote the set of probability distributions on $V(\mc{F})$, that is, the set of functions $\mu: V(\mc{F}) \to [0,1]$ such that $\sum_{v \in V(\mc{F})}\mu(v)=1$. We call a pair $(\mc{F},\mu)$, where $\mu \in \mc{M}(\mc{F})$, a \emph{weighted  graph}.  Two weighted graphs $(\mc{F},\mu)$ and $(\mc{F}',\mu')$ are \emph{isomorphic} if there exists an isomorphism $\varphi: V(\mc{F}) \to V(\mc{F}')$ between $\mc{F}$ and $\mc{F}'$ such that $\mu'(\varphi(v))=\mu(v)$ for every $v \in V(\mc{F})$. As in the case of unweighted graphs, we generally do not distinguish between isomorphic weighted graphs. 

We define \emph{the density $\lambda(\mc{F},\mu)$ of a weighted graph  $(\mc{F},\mu)$}, by  $$\lambda(\mc{F},\mu):=\sum_{F \in \mc{F} }{\prod_{v \in F}\mu(v)}.$$ The \emph{Lagrangian $\lambda(\mc{F})$ of an $r$-graph $\mc{F}$} is defined by $$\lambda(\mc{F}):=\max_{\mu \in \mc{M}(\mc{F})}\lambda(\mc{F},\mu).$$ For a family of $r$-graphs $\mf{F}$, let $\lambda(\mf{F}):=\sup_{\mc{F} \in \mf{F}}\lambda(\mc{F})$.

If an $r$-graph $\mc{F}'$ is obtained from an $r$-graph $\mc{F}$ by cloning a vertex $u \in V(\mc{F})$ to a set $W$, $\mu \in \mc{M}(\mc{F})$, $\mu' \in \mc{M}(\mc{F'})$, then we say that $(\mc{F}', \mu')$ is \emph{a one vertex blowup of 
$(\mc{F}, \mu)$}, if $\mu(v)=\mu'(v)$ for all $v \in V(\mc{F}) \setminus \{u\}$ and $\mu(u)=\sum_{w \in W}\mu'(w)$. We say that  $(\mc{F}', \mu')$ is \emph{a blowup} of $(\mc{F}, \mu)$ if  $(\mc{F}', \mu')$ is isomorphic to a weighted $r$-graph which can be obtained from $(\mc{F}, \mu)$ by repeatedly taking one vertex blowups. We denote by $\mf{B}(\mc{F},\mu)$ the family of weighted graphs isomorphic to the blowups of $(\mc{F},\mu)$. 

Next we  define the distance between weighted graphs. If $\mc{F}_1,\mc{F}_2$ are two $r$-graphs such that $V(\mc{F}_1)=V(\mc{F}_2)$
and $\mu \in \mc{M}(\mc{F}_1)(=\mc{M}(\mc{F}_2))$, we define
$$d'(\mc{F}_1,\mc{F}_2,\mu):=\sum_{F \in \mc{F}_1 \triangle \mc{F}_2}\prod_{v\in F}{\mu(v)}.$$ 
 We define \emph{the distance between general weighted $r$-graphs $(\mc{F}_1, \mu_1)$ and $(\mc{F}_2, \mu_2)$}, as $$d((\mc{F}_1, \mu_1), (\mc{F}_2, \mu_2)):=\inf d'(\mc{F}'_1,\mc{F}'_2,\mu),$$
where the infimum is taken over all $r$-graphs $\mc{F}'_1,\mc{F}'_2,$ with  $V(\mc{F}'_1)=V(\mc{F}'_2)$ and  $\mu \in \mc{M}(\mc{F}_1')=\mc{M}(\mc{F}_2')$ satisfying $(\mc{F}'_i,\mu) \in \mf{B}(\mc{F}_i,\mu_i)$ for $i=1,2$. If $(\mc{F}, \mu)$ is a weighted $r$-graph  and $\mf{F}$ is a family of $r$-graphs we define \emph{the distance  from  $(\mc{F}, \mu)$ to $\mf{F}$} as $$d_{\mf{F}}(\mc{F}, \mu) =\inf_{\mc{F'} \in \mf{F}, \mu' \in \mc{M}(\mc{F}')}d((\mc{F}, \mu), (\mc{F}', \mu')).$$ 
We  say that $\mf{F}$ is $\mf{H}$-\emph{weakly weight stable} if for every $\eps>0$ there exists $\delta>0$ such that for every $\mc{F} \in \mf{F}$ and $\mu \in \mc{M}(\mc{F})$ if $\lambda(\mc{F},\mu) \geq \lambda(\mf{H})-\delta$, then  $d_{\mf{H}}(\mc{F}, \mu) \leq \eps$. 
The next result shows that a combination of weighted stability for a restriction of $\mf{F}$ and local stability implies stability for clonable families.

\begin{theorem}[{\cite[Theorem 5.1]{gentriangle}}]\label{stabilityfromlocalstability} Let \(\mf{F}, \mf{H}\) be two clonable families of \(r\)-graphs. Let \(\mf{F}^*\) consist of all \(r\)-graphs in \(\mf{F}\) that cover pairs. If \(\mf{F}^*\) is \(\mf{H}\)-weakly weight stable, \(\mf{F}\) is \(\mf{H}\)-locally stable, then \(\mf{F}\) is \(\mf{H}\)-stable.
\end{theorem}

\subsection{Generic part of the proof}
\label{sec:genoutline}

In this subsection we present the part of the proof which is common to Theorems~\ref{thm:trees} and~\ref{thm:mubayigen} and is likely to be applicable to other problems on Tur\'an numbers of extensions. 
Let $\mf{F}$ be an $r$-graph family defined by a collection of forbidden subgraphs. Our general goal is to determine $m(\mf{F},n)$ for large $n$ by showing that $m(\mf{F},n)=m(\mf{H},n)$ , where  $\mf{H} \subseteq \mf{F}$ is a structured clonable family of conjectured extremal  examples. In particular, it suffices to show that $\mf{F}$ is $\mf{H}$-stable.

If  \(\mf{F}\) is clonable then by Theorem~\ref{stabilityfromlocalstability} it suffices to prove that \(\mf{F}\) is $\mf{H}$-locally stable  and that the subfamily \(\mf{F}^*\) is $\mf{H}$-weakly weight stable. 
However, the family $\mf{F}$ is typically not clonable. In this section we  overcome this obstacle. 

The trick is to consider a clonable subfamily of  \(\mf{F}\) instead. We define \emph{the core of $\mf{F}$} to be the maximum subfamily of \(\mf{F}\) closed under blowups: $\brm{core}(\mf{F}) = \{\mc{F}\in \mf{F}\:|\: \mf{B}(\mc{F}) \subseteq \mf{F}\}.$

We need the following  corollary of the Hypergraph Removal Lemma by R\H{o}dl, Skokan~\cite{removallem} and a classical result of Erd\H{o}s~\cite{erdos}. 

\begin{lem}[\cite{removallem}] \label{removallem} For every \(r\)-graph \(\mc{G}\) and \(\eps>0\) there exists \(\delta>0\) such that every \(r\)-graph on \(n\) vertices which contains at most \(\delta n^{\pl{v}(\mc{G})}\) copies of \(\mc{G}\) can be made \(\mc{G}\)-free by removing at most \(\eps n^{r}\) edges. 
\end{lem}

\begin{corollary}[\cite{erdos}]\label{cor:erdos} For every \(r\)-graph \(\mc{H}\), a blowup $\mc{B} \in \mf{B}(\mc{H})$ of $\mc{H}$ and $\delta >0$ there exists  \(n_0\) such that any \(r\)-graph on \(n\geq n_0\) vertices that does not contain  \(\mc{B}\) contains at most $\delta n^{\pl{v}(\mc{H})}$ many copies of \(\mc{H}\).
\end{corollary}

For a family $\mf{G}$, let $\brm{Forb}(\mf{G})$ denote the family of all $\mf{G}$-free $r$-graphs.
We deduce the following. 

\begin{lem}\label{stabcore} Let \(\mf{G}\) be a finite family of \(r\)-graphs, and let \(\mf{F}=\brm{Forb}(\mf{G})\). Then for every \(\eps>0\) there exists \(n_0\in \mathbb{N}\) such that for every \(\mc{F}\in\mf{F}\) with \(\pl{v}(\mc{F}) = n\geq n_0\) there exists \(\mc{F}'\in \brm{core}(\mf{F})\) with \(\mc{F}'\subseteq \mc{F}\) such that 
\[|\mc{F}'|\geq |\mc{F}| - \eps n^r.\]
\end{lem}
\begin{proof} 
Let $\mf{H}$ be the family of all minimal graphs $\mc{H}$ such that $\mf{B}(\mc{H}) \not \subseteq \mf{F}$. It is easy to see that $\mf{H}$ is finite. In particular every element of $\mf{H}$ is a subgraph of some graph in $\mf{G}$: For every $\mc{H} \in \mf{H}$ there exists $\mc{B}_\mc{H} \in \mf{G} \cap \mf{B}(\mc{H})$.

By Lemma~\ref{removallem} there exists $\delta>0$ such that for every $\mc{H} \in \mf{H}$ every $r$-graph on $n$ vertices  which contains at most $\delta n^{\pl{v}(\mc{H})}$ copies of $\mc{H}$ can be made $\mc{H}$-free by removing at most $\frac{\eps}{|\mf{H}|}n^r$ edges.
By Corollary~\ref{cor:erdos} there exists \(n_0\) such that for every $n \geq n_0$ and every  $\mc{H} \in \mf{H}$ every  \(\mc{B}_{\mc{H}}\)-free graph $\mc{F}$ on \(n \geq n_0\) vertices contains at most \(\delta n^{\pl{v}(\mc{H})}\) of copies of \(\mc{H}\).   Hence, by removing at most  \(\eps n^r\) edges from any graph $\mc{F} \in \mf{F}$  on $n \geq n_0$ vertices, we can obtain a subgraph $\mc{F}'$ of \(\mc{F}\), which is $\mf{H}$-free. We have $\mc{F}' \in \brm{core}(\mf{F})$, as desired. 
\end{proof}

The following result establishes the desired connection between the stability of the family \(\mf{F}\) and the stability of the \(\brm{core}(\mf{F})\). The proof modulo Lemma~\ref{stabcore} is identical to the proof of~\cite[Theorem 8.1]{gentriangle}  and we omit it.

\begin{theorem}\label{corestability} Let $\mf{G},\mf{H}$ be families of \(r\)-graphs, such that $\mf{G}$ is finite, and let \(\mf{F} = \brm{Forb}(\mf{G})\). If \(\brm{core}(\mf{F})\) is \(\mf{H}\)-stable and \(\mf{F}\) is \(\mf{H}\)-locally stable, then \(\mf{F}\) is \(\mf{H}\)-stable.
\end{theorem}

We are interested in the case when \(\mf{F} =\brm{Forb}(\brm{Ext}(\mc{G}))\) in the theorem above.  A \emph{weak extension} of an \(r\)-graph \(\mc{G}\) is an \(r\)-graph obtained from \(\mc{G}\) by adding a new edge through every uncovered pair of vertices which could contain up to \((r-2)\) new vertices. Note that in particular, $\brm{Ext}(\mc{G})$ is a weak extension of \(\mc{G}\). We denote by \(\brm{WExt}(\mc{G})\) the family of all weak extensions of the graph \(\mc{G}\). We omit the proof of the following easy lemma.

\begin{lem}\label{factcore}Let \(\mc{G}\) be an $r$ graph. Then
\begin{enumerate}
	\item $\brm{Forb}(\brm{WExt}(\mc{G}))=\brm{core}(\brm{Forb}(\brm{Ext}(\mc{G}))),$ and
	\item if  $\mf{F^*}$ is a family of all graphs in $\brm{Forb}(\brm{WExt}(\mc{G}))$ which cover pairs,  then $\mf{F^*} \subseteq \brm{Forb}(\mc{G})$.
\end{enumerate}	
\end{lem}  

We combine the consequences of Theorems~\ref{localfromvertexlocal},~\ref{stabilityfromlocalstability}, ~\ref{corestability} and Lemma~\ref{factcore} that we are interested in  into a single result as follows.

\begin{corollary}\label{cor:main} Let $\mc{G}$ be an $r$-graph, let $\mf{F}=\brm{Forb}(\brm{Ext}(\mc{G}))$. Let $\mf{F}^*$ be the family of $r$-graphs in $\brm{Forb}(\mc{G})$ that cover pairs, and let $\mf{H} \subseteq \mf{F}$ be a clonable family of $r$-graphs. If the following conditions hold
	\begin{description}
		\item[(C1)] $\mf{F}$ is  $\mf{H}$-vertex locally stable,
		\item[(C2)] $\mf{F}^*$ is $\mf{H}$-weakly weight stable,
	\end{description}
then $\mf{F}$ is $\mf{H}$-stable. In particular, there exists $n_0 \in \bb{N}$ such that if $\mc{F} \in \mf{F}$ satisfies $\pl{v}(\mc{F})=n$ and  $|\mc{F}| =m(\mf{F},n)$ for some $n \geq n_0$ then $\mc{F} \in \mf{H}$.	
\end{corollary}

\begin{proof}
From (C1) and Theorem~\ref{localfromvertexlocal} it follows that $\mf{F}$ is  $\mf{H}$-locally stable. In particular, $\mf{F}'=\brm{core}(\mf{F})$ is $\mf{H}$-locally stable. By Lemma~\ref{factcore},  $\mf{F}^*$ is exactly the family of  graphs in $\mf{F}'$ which cover pairs. Thus $\mf{F}'$ is 	$\mf{H}$-stable by Theorem~\ref{stabilityfromlocalstability}. Finally, it follows from Theorem~\ref{corestability} that $\mf{F}$ is 	$\mf{H}$- stable.
\end{proof}
	
It follows from Corollary~\ref{cor:main} that Theorems~\ref{thm:trees} and~\ref{thm:mubayigen}  will follow if verify that conditions (C1) and (C2) hold when \begin{description}
	\item[(T1)] $\mc{G}$ is an $(r-2)$-expansion of a sufficiently large Erd\H{o}s-S\'os tree $T$  and $\mf{H}=\mf{B}(\mc{K}^{(r)}_{t+r-3})$, and
	\item[(T2)] $\mc{G} = \mc{F}^{+(t+1)}$ and $\mf{H}=\mf{B}(\mc{K}^{(r)}_{t})$,
\end{description}
respectively.

In Section~\ref{sec:localstab} we describe a large family of graphs $\mc{G}$ such that $\brm{Forb}(\mc{G})$ is $\mf{B}(\mc{K}^{(r)}_{p})$-vertex locally stable (and thus $\mf{B}(\mc{K}^{(r)}_{p})$-locally stable) for some $p$ such that $\mf{B}(\mc{K}^{(r)}_{p}) \subseteq \brm{Forb}(\mc{G})$. This family will be rich enough to verify that Corollary~\ref{cor:main} (C1) holds in both cases that we are interested in. We show this in Sections \ref{sec:trees} and ~\ref{sec:mubayipikh}, respectively, where we additionally verify that the condition 
(C2) holds for $\mc{G}$ and $\mf{H}$ in (T1) and (T2).

\section{Local Stability with respect to blowups of complete graphs}
\label{sec:localstab}

We say that an \(r\)-graph \(\mc{H}\) is \emph{strongly \(t\)-colorable}, if the vertices of \(\mc{H}\) can be colored in $t$ colors such that every edge contains no two vertices of the same color. Equivalently, $\mc{H}$ is strongly $t$-colorable if and only if $\mf{B}(\mc{K}^{(r)}_t)$ is not $\mc{H}$-free.
Recall that an \(r\)-graph \(\mc{H}\)  is \(t\)-\emph{colorable}  the vertices of \(\mc{H}\) can be colored  in $t$ colors such that no edge is monochromatic. For \(r=2\) the definitions of strong $t$-colorability  and $t$-colorability  coincide, but for \(r\geq 3\) they differ.  

We say that \(\mc{H}\) is \(t\)-\emph{critical} if \(\mc{H}\) is not  strongly \(t\)-colorable, but there exists an edge \(F\in\mc{H}\) such that \(\mc{H}\setminus F\) is  strongly \(t\)-colorable; such an edge \(F\) is called \emph{critical}.  We are interested in a subfamily of \(t\)-critical \(r\)-graphs. We say that a pair $(\mc{H}, F)$ is \emph{freely \(t\)-critical}, 
if $F \in \mc{H}$ is critical and  
\((r-2)\) vertices of $F$ not contained in any other edge of $\mc{H}$. We say that these $r-2$ vertices are \emph{free in $(\mc{H}, F)$} and the other two vertices of \(F\) are \emph{critical in $(\mc{H}, F)$}. Recall that a set of edges in graph is a \emph{matching} if no two of them share a vertex.

Given graph \(\mc{H}\) and $F \in \mc{H}$ and \(v\in F\), we say that a triple $(\mc{H},F,v)$ is a \emph{\(t\)-spike}  the following conditions hold.
\begin{itemize}
\item [(i)]	 The pair $(\mc{H},F)$ is freely \(t\)-critical, and $v$ is critical in $(\mc{H},F)$.
\item [(ii)] The link \(L_{\mc{H}}(v)\) of \(v\) is a matching.
\item [(iii)] For every \(\mc{S}\subseteq [t]^{(r-1)}\), such that \(|\mc{S}|\geq {t-1 \choose r-1}\), and $\mc{S}$  is not isomorphic to $\mc{K}_{t-1}^{(r-1)}$, there exists \(\varphi: V(\mc{H}) \rightarrow [t]\) such that 
\begin{itemize} \item [(iii-a)] for every edge \(F'\in \mc{H}\setminus v\), \(|\varphi(F')| =r\) (in other words, \(\varphi\) maps the edges of \(\mathcal{H} \setminus  v\) into $[t]$ injectively),
\item [(iii-b)]  \(|\varphi(I)| = r-1\) and \(\varphi(I)\in \mc{S}\) for every \(I\in L_{\mc{H}}(v)\).
\end{itemize}
\end{itemize}

We say that $\mc{H}$ is \emph{sharply $t$-critical} is there exist $v \in F \in \mc{H}$ such that $(\mc{H},F,v)$ is a $t$-spike.
Note that for $2$-graphs the technical definition above simplifies considerably.
Indeed if $\mc{H}$ is a $2$-graph which is not $t$-colorable, and $v \in F \in \mc{H}$ are such that $\mc{H} \setminus F$  is $t$-colorable, then $(\mc{H},F,v)$ is a $t$-spike, as the conditions (ii) and (iii) above are trivially satisfied.

We are now ready to state the main result of this section.

\begin{theorem}\label{vertexlocalstab} If an $r$-graph $\mc{H}$ is sharply $t$-critical then \(\brm{Forb}(\mc{H})\) is \(\mf{B}(\mc{K}_t^{(r)})\)-vertex locally stable.
\end{theorem}

Theorem~\ref{vertexlocalstab} and the remark preceding it imply that for every $t$-critical $2$-graph $\mc{H}$ the family $\brm{Forb}(\mc{H})$ is \(\mf{B}(\mc{K}_t)\)-vertex locally stable. By Corollary~\ref{cor:main} this implies a classical theorem of Simonovits~\cite{simonovits}, which using our language can be stated as follows.

\begin{corollary}[\cite{simonovits}] Let $\mc{H}$ be a $t$-critical $2$-graph.  Then \(\brm{Forb}(\mc{H})\) is \(\mf{B}(\mc{K}_t)\)-stable.
\end{corollary}

The rest of the section is occupied by the proof of Theorem~\ref{vertexlocalstab}.

Let  \(\pl{e}(t, r) := t^{-r}{t \choose r},\) \(\pl{d}(t, r) := t^{-(r-1)}{t-1 \choose r-1}.\) Note that $\pl{e}(t, r) $ stands for the normalized edge density of \(\mc{K}_{t}^{(r)}\), and \(\pl{d}(t, r) \) is the normalized degree of a vertex in \(\mc{K}_{t}^{(r)}\). For an $r$-graph $\mc{F}$ and a partition $\mc{P}$ of $V(\mc{F})$ we say that $F \in \mc{F}$ is \emph{$\mc{P}$-transversal} if $F$ intersects each part of $\mc{P}$ in at most one element. 
Let \(\mc{B}\in\mf{B}(\mc{K}_t^{(r)})\) with \(\pl{v}(\mc{B} )  = n\) and  let \(\mc{P} = \{P_1, P_2, \dots, P_t\}\) be \emph{the blowup partition of \(\mc{B}\)}, that is a partition of $V(\mc{B})$ such that every edge of $\mc{B}$ is $\mc{P}$-transversal.  We say that \(\mc{B}\) is \(\eps\)-balanced if for every \(i = 1, 2, \dots, t\),
\[\left|P_i - \frac{n}{t}\right|\leq \eps.\]
 We omit the proof of the following routine lemma.
\begin{lem}
\label{sizelem} For every $\eps >0$ there exists $\delta >0$ and $n_0\in \mathbb{N}$ such that the following holds. If $\mc{B}\in\mf{B}(\mc{K}_t^{(r)})$ with $\pl{v}(\mc{B})=n\geq n_0$ and $|\mc{B}|\geq \left(\pl{e}(t,r) - \delta \right) n^r$, then \(\mc{B}\) is \(\eps\)-balanced.
\end{lem}

We also need  the following two auxilliary lemmas.

\begin{lem}\label{transversal}For given \(t\geq r \geq 2\), let $(\mc{H},F)$ be freely \(t\)-critical. Then there exist \(\eps>0\) and \(n_0\in \mathbb{N}\) such that the following holds. Let \(\mc{F}\) be an \(\mc{H}\)-free \(r\)-graph with \(\pl{v}(\mc{F}) =n\geq n_0\) vertices, and let \(\mc{B}\in\mf{B}(\mc{K}_t^{(r)})\) with \(\pl{v}(\mc{B}) = n\) and the blowup partition \(\mc{P} = \{P_1, P_2, \dots, P_t\}\). If  \(|L_{\mc{F}}(v) \triangle L_{\mc{B}}(v)| \leq \eps n^{r-1}\) for every \(v\in V(\mc{F})\), and \(|\mc{F}|\geq (\pl{e}(t,r) - \eps) n^r\), then \(\mc{F}\) is \(\mc{P}\)-transversal. Moreover, if \(\mc{F}'\) is an \(\mc{H}\)-free \(r\)-graph such that \(\mc{F}\subseteq \mc{F}'\), then \(\mc{F}'\) is \(\mc{P}\)-transversal.
\end{lem}
\begin{proof}
Choose $\eps_{\ref{sizelem}} \ll \min\{ \frac{1}{t},\frac{1}{\pl{v}(H)}\}$, and let \(\delta_{\ref{sizelem}}<\eps_{\ref{sizelem}} \) be chosen to satisfy  Lemma~\ref{sizelem} applied with \(\eps = \eps_{\ref{sizelem}}\).  Choose $\eps \ll \min\{\delta_{\ref{sizelem}},\frac{1}{m}\}$.
We have
\[|\mc{F}\triangle \mc{B} |  = \frac{1}{r} \sum_{v\in V(\mc{F})}{|L_{\mc{F}}(v) \triangle L_{\mc{B}}(v) |} \leq \frac{\eps}{r} n^r.\]
Therefore
\[|\mc{B}| \geq |\mc{F} | - |\mc{F} \triangle \mc{B}| \geq \left(\pl{e}(t,r) - \eps \left(1+ \frac{1}{r}\right)\right) n^r \geq (\pl{e}(t,r)  - \delta_{\ref{sizelem}}) n^r ,\]
and Lemma ~\ref{sizelem} implies that \(\mc{B}\) is \(\eps_{\ref{sizelem}}\)-balanced.

It suffices to verify only the last conclusion of the lemma. We assume, for a contradiction, that there exists a non-transversal edge \(F \in \mc{F'}\), with \(v_1,v_2\in F\cap {P_j}\) for some  \(j\). Let $m =|V(H)|+r-2$.
We will  show that \(\mc{F}\) contains a copy of  the complete \(t\)-partite $r$-graph with \(m\) vertices in each part, with $v_1$ and $v_2$ lying in the same part of this copy. Together with \(F\) this copy will induce a copy \(\mc{H}\), yielding the desired contradiction.  Without loss of generality, assume that \(j=1\).

Sample \(m\) distinct vertices from  \(P_i \setminus F\) uniformly at random for every \(i=2, \dots, t\), and sample $m-2$ vertices from $P_1 \setminus F$. Let \(\mc{R}\) be the subgraph of \(\mc{F}\) induced by these vertices, $v_1$ and $v_2$. It suffices to show that with a positive probability every $\mc{P}$-transversal $r$-tuple $I \subseteq V(\mc{R})$ is an edge of $\mc{F}$.
Therefore it is enough to show that 
\begin{itemize}
	\item If $I$ is a set of $r-1$ vertices sampled uniformly at random from distinct parts of $\mc{P} -\{P_1\}$ then $$\brm{P}[I \cup \{v_i\} \not \in \mc{F}] < \frac{1}{4t^{r-1}m^{r-1}}$$
	for $i=1,2$, and
	\item If $I$ is a set of $r$ vertices sampled uniformly at random from distinct parts of $\mc{P}$ then $$\brm{P}[I \not \in \mc{F}] < \frac{1}{4t^{r}m^{r}}$$
	for $i=1,2$.  
\end{itemize}
Both statements are routine. As $|L_{\mc{F}}(v_i) \triangle L_{\mc{B}}(v_i)| \leq \eps n^{r-1}$ and $\mc{B}$ is \(\eps_{\ref{sizelem}}\)-balanced, it follows that  $$1 - \frac{|L_{\mc{F}}(v_i) \cap L_{\mc{B}}(v_i)|}{|L_{\mc{B}}(v_i)|} \ll \frac{1}{t^{r}m^{r}},$$
and thus the probability that a transversal $r$-tuple containing $v_i$ is not in $\mc{F}$ is sufficiently small. As $|\mc{F}\triangle \mc{B} | \ll \frac{n^r}{t^{r}m^{r}}$, the second statement similarly follows.
\end{proof}

\begin{lem}\label{links} Let $(\mc{H},F,v)$ be a  \(t\)-spike. For all \(\eps>0\) there exists \(\delta >0\) and \(n_0\in \mathbb{N}\) such that the following holds. If \(\mc{F}\) is an \(\mc{H}\)-free \(r\)-graph with \(\pl{v}(\mc{F})=n\geq n_0\), \(d_{\mf{B}}(\mc{F}) \leq \delta n^r\), \(|\mc{F}|\geq (\pl{e}(t,r) - \delta)n^r\) ,  and \(|L_{\mc{F}}(v)| \geq (d(t,r) - \delta) n^{r-1}\) for every \(v\in V(\mc{F})\), then there exists \(\mc{B}_0\in \mf{B}(\mc{K}_{t}^{(r)})\) with \(\pl{v}(\mc{B}_0) = n\) such that for every \(v\in V(\mc{F})\)
\[|L_{\mc{F}}(v) \triangle L_{\mc{B}_0}(v)|\leq \eps n^{r-1}. \] 
\end{lem}

\begin{proof}Let $\eps_{\ref{transversal}}$ be chosen to satisfy Lemma~\ref{transversal}. We choose $$ 0<\delta\ll\eps_{\ref{sizelem}} \ll \gamma \ll \beta\ll\min\left\{\eps_{\ref{transversal}},\eps,\frac{1}{\pl{v}(\mc{H})}\right\}$$ to satisfy the constraints appearing further in the proof.
Let $\delta_{\ref{sizelem}}$ be chosen to satisfy Lemma~\ref{sizelem} applied with $\eps=\eps_{\ref{sizelem}}$. In particular we choose \(\delta\) such that 
$\delta \ll \delta_{\ref{sizelem}}$.

Let \(\mc{B}\in \mf{B}\) be such that \(|\mc{F}\triangle \mc{B}| = d_{\mf{B}}(\mc{F})\), and let \(\mc{P}=\{P_1, P_2, \dots, P_t\}\) be the blowup partition of \(\mc{B}\).  Since 
\[|\mc{F}|\geq (\pl{e}(t,r) - \delta) n^r \geq (\pl{e}(t,r) - \delta_{\ref{sizelem}}) n^r ,\]
it follows that \(\mc{B}\) is \(\eps_{\ref{sizelem}}\)-balanced by Lemma~\ref{sizelem}. Consider the set 
\[J:=\{v\in V(\mc{F})| |L_{\mc{F}}(v) \triangle L_{\mc{B}}(v)| > \gamma n^{r-1}\}.\]
Let $\delta_1 = \frac{\delta r}{\gamma}$. It is easy to see that \(|J| \leq \delta_1 n\).  

Let $\mathcal{F}' := \mathcal{F}|_{V(\mc{F})\setminus J}$, $n'=\pl{v}(\mc{F}')$, $\mathcal{B}':= \mathcal{B}|_{V(\mc{F})\setminus J}$, $P_j':=P_j\setminus J$ for each $j\in[t]$, and $\mc{P}'=\{P'_1, P'_2, \dots, P'_t\}$. The graph \(\mc{F}'\) satisfies the assumptions of Lemma~\ref{transversal}. Indeed, for every \(v\in V(\mc{F}')\),
\[| L_{\mc{F}'}(v)\triangle  L_{\mc{B}'}(v)| \leq \gamma n^{r-1} \leq\eps_{\ref{transversal}} (1-\delta_1)^{r-1} n^{r-1}
\leq
\eps_{\ref{transversal}} (n')^{r-1}.\]
Similarly,
$|\mc{F}'|\geq (\pl{e}(t,r) -\eps_{\ref{transversal}}) (n')^{r-1} 
$.
Thus  $\mc{F}$ is $\mc{P}$-transversal by Lemma~\ref{transversal}.

 Our next goal is to extend \(\mc{B}'\) to a blowup \(\mc{B}_0\) of $\mc{K}_t^{(r)}$ with $V(\mc{B}_0)=V(\mc{F})$, as follows. For each $u\in J$  we will find a unique index $i_{u} \in [t]$, such that $u$ behaves as the vertices in the partition class $P_{i_u}'$, and add the vertex \(u\) to this partition class. 

Consider \(u \in J\).
For $I \subseteq [t]$, let $$E_I(u):=\{F \in \mc{F} \:  | u\in F, \: |F \cap P'_i| =1 \: \mathrm{for\: every}\: i \in I  \}.$$
We construct an auxiliary   $(r-1)$-graph  $\mc{S}=\mc{S}(u)$ with $V(\mc{S})=[t]$ such that $I\in \mc{S}$ if and only if $\left|E_I(u)\right|\geq \beta n^{r-1}.$
We aim to show that there exists a unique $j_u \in [t]$ such that $\mc{S}$ is isomorphic to the link graph of $j_u$ in $\mc{K}_{t}^{(r)}$. 
It is easy to see that  \(|\mc{S}| \geq d(t,r) t^{r-1} = {t-1 \choose r-1}\), as long as \(\eps_{\ref{sizelem}}\), \(\delta\), \(\delta_1\) and \(\beta\) are sufficiently small compared to \(\frac{1}{t^r}\). 

\begin{claim}\label{link}\(\mc{S}\) is isomorphic to $\mc{K}_{t-1}^{(r-1)}$.
\end{claim}

\begin{proof} Suppose not. Let  \(\varphi: V(\mc{H})\rightarrow [t]\) be as in the definition of the $t$-spike $(H,F,v)$. Let $\psi: V(\mc{H}) \rightarrow V(\mc{F})$ be a random map such that $\psi(v)=u$, and  let $\psi(w)$ be chosen uniformly at random in $P_{\varphi(w)}$ for every $w \in V(\mc{H}) \setminus v$. We will show that, with probability bounded away from zero as a function of $\beta$ and independent on $n$, the map $\psi$ maps all edges of $\mc{H}$ to edges of $\mc{F}$. It will follow that $\mc{F}$ is not $\mc{H}$-free yielding the desired contradiction. If $I \in \mc{H}$, $v \not \in I$ then $\brm{P}[ \psi(I) \notin \mc{F}] \leq 2\gamma t^r$ , as in Lemma~\ref{transversal}. Thus, $$\brm{P}[ \psi(I) \notin \mc{F} \; \mathrm{for\; some} \; I\in \mc{H} \; \mathrm{such \;that} \; v \not \in I ] \leq 2\gamma \pl{v}(\mc{H})^r t^r \ll  \beta.$$ 
If $I \in L_{\mc{H}}(v)$ then $\brm{P}[\psi(I \cup \{v\}) \in \mc{F}] \geq \beta$. As $L_{\mc{H}}(v)$ is a matching it follows that the events $\{\psi(I \cup \{v\} \}) \in \mc{F} \}_{I  \in L_{\mc{H}}(v)}$ are independent. Thus $$\brm{P}[\psi(I) \in \mc{F}\; \mathrm{for\; every} \; I \in L_{\mc{H}}(v)] \geq \beta^{| L_{\mc{H}}(v)|}.$$
The desired conclusion follows. 
\end{proof}

By Claim~\ref{link}, for every \(u\in J\) there exists unique \(i_u\in [t]\) such that for every \(I\in {t \choose r-1}\) with \(i_u\in I\), \(|E_{I}(u)| < \beta n^{r-1}\). Now we are ready to extend the blowup \(\mc{B}'\) to a blowup with partition \(\mc{P}^{*}\) as following. For every \(i\in [t]\), define $P_{i}^{*}:=P_i'\cup \{u\in J \: |\:  i_u = i\}.$ Let $\mc{B}_0 \supseteq \mc{B}'$ be the blowup of $\mc{K}_{t}^{(r)}$ with the blowup partition \(\mc{P}^*=\{P_1^*,\ldots,P_t^*\}\).  
 It remains to show that
$| L_{\mc{B}_{0}}{(v)}\triangle  L_{\mc{F}}(v)|\leq \eps n^{r-1}$
for every \(v\in V(\mc{F})\).
For each $v\in V(\mc{F})\setminus J$, we have 
\begin{align*}| L_{\mc{B}_{0}}{(v)}\triangle  L_{\mc{F}}(v)|&\leq | L_{\mc{B}'}{(v)}\triangle  L_{\mc{F'}}(v)| + |J|n^{r-2}\\ &\leq \gamma n^{r-1} + \delta_1 n^{r-1} \leq \eps n^{r-1}.
\end{align*}
We now consider $v\in J$. As we observed earlier, $\mc{F}$ is $\mc{P}'$-transversal, therefore for every $F \in  L_{\mc{F} \setminus \mc{B}_0}(v)$, either either $F$ contains a vertex from \(J\) or there exists $I \in [t]^{(r-1)}$, $I \notin \mc{S}(v)$ such that $F \cup \{v\} \in E_I(v)$. Thus,
$$ | L_{\mc{F} \setminus \mc{B}_0}(v) |
 \leq \left(\delta_1 + { t-1\choose r-2} \beta\right) n^{r-1} \leq \frac{\eps}{4}n^{r-1}.
$$
Finally,
\begin{align*} |& L_{\mc{F}}(v)\triangle  L_{\mc{B}_0}(v)| = 
2|  L_{\mc{F} \setminus\mc{B}_0}(v) | + | L_{\mc{B_0}}(v)| - | L_{\mc{F}}(v)| \\ &\leq \frac{\eps}{2}n^{r-1} +t^{r-1}\pl{d}(t,r) \left(\frac{1}{t}+\eps_{\ref{sizelem}}+\delta_1\right)^{r-1}n^{r-1}-(\pl{d}(t,r)-\delta)n^{r-1}\\&\leq \eps n^{r-1},
\end{align*}
as desired.
\end{proof}

\begin{proof}[Proof of Theorem~\ref{vertexlocalstab}:] Let $(\mc{H},F,v)$ be a  \(t\)-spike, and let \(\mf{B}= \mf{B}(\mc{K}_t^{(r)})\). We want to show that there exist \(\eps, \alpha, n_0>0\) such that for every \(\mc{F}\ \in \brm{Forb}(\mc{H})\) with \(\pl{v}(\mc{F}) = n\geq n_0\), such that \(d_{\mf{B}}(\mc{F}) \leq \eps n^r\),  and \(|L_{\mc{F}}(v)|\geq (\pl{d}(t,r) - \eps) n^{r-1}\) for every \(v\in V(\mc{F})\), we have 
$|\mc{F}|\leq m(\mf{B}, n) - \alpha d_{\mf{B}}(\mc{F}).$

Let \(\eps_{\ref{transversal}}\) be chosen to satisfy Lemma~\ref{transversal}. Let \(\delta_{\ref{links}}\) be chosen to satisfy Lemma~\ref{links}, applied with \(\eps = \eps_{\ref{transversal}} \).  We show that \(\eps = \min\{\frac{\eps_{\ref{transversal}}}{r}, \frac{\delta_{\ref{links}}}{2}\}\) and \(\alpha = 1\) satisfy the desired conditions. We  assume that 
\[|\mc{F}|\geq (\pl{e}(t,r) - 2\eps) n^r \geq  (\pl{e}(t,r) - \delta_{\ref{links}}) n^r,\]
since otherwise the result holds. Thus by Lemma~\ref{links}, there exists some \(\mc{B}\in \mf{B}\) such that for every \(v\in V(\mc{F})\) we have 
$|L_{\mc{F}}(v) \triangle L_{\mc{B}}(v)| \leq \eps_{\ref{transversal}} n^{r-1}. $
Therefore \(\mc{F} \subseteq \mc{B}\) by Lemma~\ref{transversal},
and $|\mc{F}| =  |\mc{B}| - |\mc{B}\setminus \mc{F}|\leq  m(\mf{B}, n) - d_{\mf{B}}(\mc{F}),
$
as desired.
\end{proof}

\section{Proof of Theorem~\ref{thm:trees}}
\label{sec:trees}

Let \(\mc{T}\) denote the \((r-2)\)-expansion of the tree \(T\). By Corollary~\ref{cor:main}, it suffices to prove that  \(\brm{Forb}(\brm{Ext}(\mc{T}))\) is \(\mf{B}(\mc{K}_{t+r-3}^{(r)})\)-vertex locally stable and \(\brm{Forb}(\mc{T})\) is \(\mf{B}(\mc{K}_{t+r-3}^{(r)})\)-weakly weight stable. By Theorem~\ref{vertexlocalstab}, the following lemma accomplishes the first step.

\begin{lem} If  \(T\) is a tree on \(t \geq 3\) vertices,  then \(\brm{Ext}(\mc{T})\) is sharply \((t+r-3)\)-critical.
\end{lem}
\begin{proof} 
Let $v$ be a leaf of $T$, and let $F$ be an edge of $\brm{Ext}(\mc{T}) - \mc{T}$ containing $v$. We show that $(\brm{Ext}(\mc{T}),F,v)$ is a $(t+r-3)$-spike. Let $u$ be the unique vertex in $(F \cap V(T)) -\{v\}$. Note that $\pl{v}(\mc{T})=t+r-2$, and every pair of vertices in $V(\mc{T})$ is covered in \(\brm{Ext}(\mc{T})\).  Condition (i) in the definition holds, as \(\brm{Ext}(\mc{T})\) is not $(t+r-3)$-colorable, but in \(\brm{Ext}(\mc{T}) - F\), one can use the same color on $u$ and $v$.

As \(v\) is adjacent to the unique vertex in \(T\), \(L_{\brm{Ext}(\mc{T})}(v)\) is a matching. It remains to verify (iii).
For \(\mc{S}\subseteq [t+r-3]^{(r-1)}\) such that \(|\mc{S}|\geq { t +r-4 \choose r-1}\) and $\mc{S}$ is not isomorphic to  $\mc{K}_{t+r-4}^{(r-1)}$ we define mapping \(\varphi : V(\brm{Ext}(\mc{T}))\rightarrow [t+ r-3]\) satisfying condition (iii) in the definition of a $t$-spike  as follows. 

Consider  the subgraph \(\mc{T}'\) of \(\brm{Ext}(\mc{T})\) induced by the vertex set  $V(\brm{Ext}(\mc{T}))$$-(V(L_{\brm{Ext}(\mc{T})}(v)) - V(\mc{T}))$. Then $\mc{T}'$ is \((t + r -3)\)-strongly colorable. Let $\varphi$ be defined on $V(\mc{T}')$ so that $\varphi$ is a strong \((t + r -3)\)-coloring of $\mc{T}'$, and moreover, $\varphi(E - \{v\}) \in \mc{S}$ for the unique edge $E \in \mc{T}'$ such that $v \in E$. It follows that (iii-a) holds for $\varphi$.

It remains to extend $\varphi$ so that it satisfies (iii-b).
For every $w  \in V(T) -\{u,v\}$ there exists a unique $I \in \mc{L}_{\brm{Ext}(\mc{T})}(v)$ such that  $w \in I$.
Since $\mc{S}$ is not isomorphic to  $\mc{K}_{t+r-4}^{(r-1)}$, there exists   \(S \in \mc{S}\) such that $\varphi(w) \in S$. We extend \(\varphi\) to $I-\{w\}$ so that $\varphi(I)=S$. Clearly, the resulting map $\varphi: V(\brm{Ext}(\mc{T})) \to [t+ r-3]$ satisfies (iii-b).
\end{proof}

The following theorem will complete the proof of Theorem~\ref{thm:trees}. 

\begin{theorem}\label{weakweightstability} For every \(r\geq2\) there exists real \(M_r \) such that, if \(T\) is an Erd\H{o}s-S\'{o}s-tree on  \(t\geq M_r\) vertices,  then  \(\brm{Forb}(\mc{T})\) is \(\mf{B}(\mc{K}_{t+r-3}^r)\)-weakly weight stable.
\end{theorem}

The proof of Theorem~\ref{weakweightstability} relies on a result by Sidorenko~\cite{sidorenko}. Its statement involves a function
$$f_r(x) = \frac{1}{(x+r-3)^r}{x+ r-3 \choose r}\frac{t-2}{x-2}.$$
Let us first note the following useful observations concerning $f_r(x)$ from~\cite{sidorenko}:
\begin{description}
	\item[(F1)] The function $f_r(x)$ is strictly decreasing for sufficiently large $x$,
	\item[(F2)]  $$f_r(x)=\frac{1}{r}\left(\frac{x+r-4}{x+r-3}\right)^{r-1}f_{r-1}(x), \qquad \mathrm{and}$$ 
	\item[(F3)] $\lambda(\mc{K}_{t+r-3}^{(r)})=f_r(t)$.
\end{description}
\begin{theorem}[{\cite[Lemma 3.3]{sidorenko}}]\label{sidorenko1} For $r \geq 2$ let $M_r$ be such that $f_r(x)$ is decreasing for $x \geq M_r$.  If \(T\) is an Erd\H{o}s-S\'{o}s-tree on  \(t\geq M_r\) vertices  then $\lambda(\mc{F}, \mu) \leq f_r(x)$ for every \(\mc{F} \in \brm{Forb}(\mc{T}), \mu\in M(\mc{F})\), 
where \(x= \max{\{t, \frac{1}{\gamma}  - r + 3\}}\) and \(\gamma=\max_{v\in V(\mc{F})}{\mu(v)}\). In particular, $\lambda(\brm{Forb}(\mc{T})) = f_r(t).$
\end{theorem}

Given an \(r\)-graph \(\mc{F}\), $u,v \in V(\mc{F})$ and  \(\mu\in \mc{M}(\mc{F})\), let $\lambda(\mc{F},\mu, u) =\sum_{I\in L_{\mc{F}}(u)}{\mu(I)},$ and let $\lambda(\mc{F},\mu, u,v) =\sum_{I\in L_{\mc{F}}({u,v})}{\mu(I)}.$ The following technical lemma is useful in the proof of Theorem~\ref{weakweightstability}.

\begin{lem}\label{sidolem}  For every \(\eps>0\) there exists \(\delta>0\) such that if \(\mc{F}\) is an \(r\)-graph, $u \in V(\mc{F})$, \(\mu\in \mc{M}(\mc{F})\), \(\lambda(\mc{F}, \mu) \geq \lambda(\mc{F})  - \delta\) and \(\mu(u)\geq \eps\), then  \(\lambda(\mc{F}, \mu, u) \geq r \lambda(\mc{F}) - \eps.\)
\end{lem}
\begin{proof} We assume without loss of generality that $\eps<1$, and let $\delta = (\eps^3-\eps^4)/r $. Suppose for a contradiction that $\lambda(\mc{F}, \mu, u) < r \lambda(\mc{F}) - \eps.$ We have 
$r \lambda (\mc{F}, \mu) = \sum_{v\in V(\mc{F})}{\mu(v) \lambda(\mc{F}, \mu, v)}.$
Thus, there exists \(u' \in V(\mc{F}) - \{u\}\) such that \(\lambda(\mc{F}, \mu, u') \geq r\lambda(\mc{F})\). 

Let \(\mu' \in \mc{M}(\mc{F})\) be defined as follows.  Let   \(\mu'(u' ) = \mu(u') + \eps^2 \), $\mu'(u)$ $=$ $\mu(u) - \eps^2$,   and let \(\mu'(v) = \mu(v)\) for every \(v \in V(\mc{F}) - \{ u, u'\}\). We have
\begin{align*} r \lambda (\mc{F}, \mu') &= \sum_{v\in V(\mc{F})}{\mu'(v) \lambda(\mc{F}, \mu', v)} 
= r \lambda (\mc{F}, \mu)  + \eps^2( \lambda(\mc{F}, \mu, u') - \lambda(\mc{F}, \mu, u) ) -\eps^4 \lambda(\mc{F}, \mu, u, u') \\
&> r \lambda (\mc{F}, \mu) +  \eps^3 - \eps^4 
\geq r \lambda(\mc{F} ) - r\delta  + \eps^3 - \eps^4
\geq r \lambda(\mc{F}), 
\end{align*}
a contradiction.
\end{proof}

In the rest of the section it will be convenient for us to occasionally consider 
subprobabilistic measures on the vertex set of a graph, rather than probabilistic ones. For a graph $\mc{F}$ let $\mc{M_{\leq 1}(\mc{F})} \supset \mc{M}(\mc{F})$ denote the set of functions $\mu: V(\mc{F}) \to \brm{R}_{+}$ such that $\mu(V(\mc{F})) \leq 1$. The density $\lambda(\mc{F}, \mu)$ and the distance $d((\mc{F},\mu),(\mc{F},\mu'))$ for  $\mu,\mu' \in \mc{M_{\leq 1}(\mc{F})}$ are defined as in Section~\ref{sec:weighted}. Moreover, it is easy to check that 
 \(|\lambda(\mc{F}, \mu) - \lambda(\mc{F}, \mu')| \leq  ||\mu-\mu'||_1\), and \(d((\mc{F}, \mu), (\mc{F}, \mu')) \leq ||\mu-\mu'||_1\) for any \(\mu, \mu'\in M_{\leq 1}(\mc{G})\).
All the technical work in the proof of Theorem~\ref{weakweightstability} is accomplished in the following lemma.

\begin{lem}\label{auxlem1} For every \(r\geq2\) there exists \(M_r\) such that if \(\mc{T}\) is the $(r-2)$-expansion of an  Erd\H{o}s-S\'{o}s-tree $T$ on  \(t\geq M_r\) vertices  then  the following holds. For every \(\eps>0\) there exists \(\delta>0\) such that if \(\mc{F}\in \brm{Forb}(\mc{T})\), \(\mu \in \mc{M}(\mc{F})\) with \(\lambda(\mc{F}, \mu) \geq \lambda(\brm{Forb}(\mc{T}))  - \delta\), then there exists  \(S\subseteq V(\mc{F})\) such that
$\mu(S) \geq 1- \eps$, and 
\(\mc{F}[S]\) is isomorphic to \(\mc{K}_{t+r-3}^{(r)}\).
\end{lem}

\begin{proof} 
The proof is by induction on \(r\). 

The base case is $r=2$. We have $\lambda(\brm{Forb}(T))=\frac{t-2}{t-1}$, and $\lambda(\mc{F})=\frac{\omega-2}{\omega-1}$, where $\omega$ is the number of vertices in the maximum complete subgraph of $\mc{F}$. It follows that 
$\omega = t-1$. However, the graph $\mc{F}$ is $T$-free, and therefore every complete subgraph of  $\mc{F}$ of size $t-1$ is a component of $\mc{F}$. Note that if $\mc{F} = \mc{F}_1 \cup \mc{F}_2$ and $V(\mc{F}_1) \cap V(\mc{F}_2) = \emptyset$ then $$\lambda(\mc{F},\mu) \leq \mu(V(\mc{F}_1))^2\lambda(\mc{F}_1) +  \mu(V(\mc{F}_2))^2\lambda(\mc{F}_2).$$
It follows that there exists a complete component $\mc{C}$ of $\mc{F}$ such that $\pl{v}(\mc{C})=t-1$ and $2\mu(V(\mc{C}))(1-\mu(V(\mc{C}))) \leq \delta.$ Therefore taking $S=V(\mc{C})$ we see that $\delta = \eps -\eps^2$ satisfies the theorem in the base case. 

We move on to the induction step. Let $M_r$ be chosen so that  $f_k(x)$ is strictly decreasing for $x > M_r$ and $k \leq r$. This choice is possible by (F1). By the induction hypothesis there exists $\delta_{r-1}>0$ such  that the claim holds for the $(r-3)$-expansion of $T$.
The parameters  $\delta \ll \delta'' \ll \eps' \ll \delta_{r-1} $ will be chosen to satisfy the inequalities (occasionally implicit) appearing further in the proof.

Let \(u\in V(\mc{F})\) be such that \(\mu(u) = \max_{v\in V(\mc{F})}{\mu(v)}\), and let \(\gamma=\mu(u)\). Suppose that $\gamma < 1/(t+r-3)$. Then by Theorem~\ref{sidorenko1} and our assumptions we have 
$
f_r(t) - \delta \leq \lambda(\mc{F}, \mu) \leq f_r(1/\gamma - r+3).
$
This inequality and the choice of $M_r$ imply that 
\begin{equation}\label{gamma}
\gamma  \geq \frac{1 -\eps'}{t+r-3},
\end{equation}
 as long as $\delta$ is 
sufficiently small compared to $\eps'$.
By Lemma~\ref{sidolem} we  have
\begin{equation}\label{lagrangineq}\lambda(\mc{F}, \mu, u) \geq r\lambda(\mc{F}) - \eps'/2 \geq rf_r(t) - \eps',
\end{equation}
once again assuming that $\delta$ is sufficiently small compared to $\eps'$ for the conditions of Lemma~\ref{sidolem} to be satisfied. 

Let $\mc{T}'$ be the \((r-3)\)-enlargement of the tree \(T\), and let $\mc{F'} = L_{\mc{F}}(u)$. Then $\mc{F}'$ is  $\mc{T}'$-free.
Let $\mu' \in \mc{M}(\mc{F'})$ be given by $\mu'(v) =\frac{\mu(v)}{1-\gamma} $, for every \(v\in  V(\mc{F}')\). Using (\ref{gamma}),(\ref{lagrangineq}) and (F2) we have 
\begin{align*}\lambda(\mc{F'}, \mu') &=  \frac{\lambda(\mc{F}, \mu, u)}{(1-\gamma)^{r-1}} 
\geq \left(\frac{t+r-3}{t+r-4 +\eps'}\right)^{r-1}(rf_r(t) - \eps')
\\ &= \left(\frac{t+r-4}{t+r-4 +\eps'}\right)^{r-1}f_{r-1}(t) - \eps'  \left(\frac{t+r-3}{t+r-4 +\eps'}\right)^{r-1}\geq
f_{r-1}(t) - \delta_{r-1},
\end{align*}
where the last inequality  holds for $\eps'$ is sufficiently small compared to $\delta_{r-1}$. By the choice of $\delta_{r-1}$ there exists   \(S'\subseteq V(\mc{F}) - \{u\}\) with \(|S'| = t+r-4\) such that \(\mu'(S') \geq 1 - \eps\). Let \(S=S'\cup \{u\}\) then
$$\mu(S) = \mu(u) + \mu'(S') (1-\mu(u))
\geq 1- \eps.
$$
It remains to show that $\mc{F}[S]$ is complete. We assume without loss of generality that $\eps$ is sufficiently small.
Let $\mc{F}^*=\mc{F}[S]$, and let $\mu^*=\mu|_S$. Then $\mu^*\in \mc{M}_{\leq 1}(\mc{F}^*)$, and $\lambda(\mc{F}^*,\mu^*) \geq \lambda(\mc{F},\mu) - \|\mu - \mu^*\| \geq f_r(t)-\delta-\eps.$ 
It follows that  $\mc{F}^*$ is complete, as long as $\eps$ and $\delta$ are sufficiently small with respect to $t$ and $r$.
\end{proof}

Lemma~\ref{auxlem1} directly implies Theorem~\ref{weakweightstability}, as follows.

\begin{proof}[Proof of Theorem~\ref{weakweightstability}:] 
Let $\mf{B} = \mf{B}(\mc{K}_{t+r-3}^{(r)})$.
For every \(\eps>0\) we need to show existence of \(\delta>0\) such that if \(\mc{F}\in \brm{Forb}(\mc{T})\), \(\mu\in \mc{M}(\mc{F})\) with  \(\lambda(\mc{F}, \mu) \geq \lambda(\brm{Forb}(\mc{T})) - \delta\), then \(d_{\mf{B}}(\mc{F},\mu) \leq \eps\). Let $\delta$ be chosen so that Lemma~\ref{auxlem1} holds. Then there exists $S \subseteq V(\mc{F})$ such that 
$\mu(S) | \geq 1 - \eps$, and $\mc{F^*}:=\mc{F}[S]$ is isomorphic to $\mc{K}_{t+r-3}^{(r)}$.  Let $\mu' \in \mc{M}_{\leq 1}(\mc{F})$ be the measure obtained from $\mu$ by setting $\mu'(v) =0$ for every $v \in V(\mc{F})-S$. 
  Then, 
\begin{align*}
d_{\mf{B}}(\mc{F},\mu) &\leq d((\mc{F}^*, \mu|_S),(\mc{F}, \mu)) = 
d((\mc{F}, \mu'),(\mc{F}, \mu)) \leq \eps
\end{align*}
as desired.
\end{proof}

\section{The proof of Theorem~\ref{thm:mubayigen}}
\label{sec:mubayipikh}

By Corollary~\ref{cor:main} and Theorem~\ref{vertexlocalstab}, Theorem~\ref{thm:mubayigen} follows immediately from the following Lemmas~\ref{lem:mubayilocal} and~\ref{lem:mubayiweight}.

\begin{lem}\label{lem:mubayilocal} If \(\mc{F}\) is an \(r\)-graph that covers pairs, then for any \( t\geq \pl{v}(\mc{F})\), the \(r\)-graph \(\brm{Ext}(\mc{F}^{+(t+1)})\) is sharply \(t\)-critical.
\end{lem}

\begin{proof} Let $\mc{H}=\brm{Ext}(\mc{F}^{+(t+1)})$. Consider  $v \in V(\mc{F}^{+(t+1)}) - V(\mc{F})$ and let $F$ be any edge containing $v$. We will show that $(\mc{H},F,v)$ is a $t$-spike. 
Conditions (i) and (ii) in the definition of a $t$-spike are easy to verify. Consider \(\mc{S}\subseteq [t]\) with \(|\mc{S}|\geq {t-1 \choose r-1}\) with  $\mc{S}$ not isomorphic to  $\mc{K}_{t-1}^{(r-1)}$. We define the map \(\varphi:V(\mc{F}) \rightarrow [t]\) satisfying the conditions (iii-a) and (iii-b) as follows. 

Consider  the subgraph \(\mc{H}'\) of \(\mc{H}\) induced by the vertex set \(V(\mc{H})- (V(\mc{L}_{\mc{H}}(v)) - V(\mc{F}^{+(t+1)}))\).  Let \(\varphi|_{V(\mc{H}')}\) be any strong $t$-coloring of $\mc{H}'$. Then (iii-a) holds. For every \(I \in \mc{L}_{\mc{H}}(v)\), let $w$ be the unique vertex in $(I \cap V(\mc{F}^{+(t+1)}))-\{v\}$. Then there exists $S \in \mc{S}$ such that $w \in S$. Extend $\varphi$ to $I - \{v,w\}$ so that $\varphi(I)=S$. The resulting map $\varphi$ satisfies (iii-b).
\end{proof}

\begin{lem}\label{lem:mubayiweight}  Let \(\mc{F}\) be an \(r\)-graph that covers pairs, and let $t$ be such that \(  t  \geq \pl{v}(\mc{F}) \) and  \(\pi(\mc{F}) < r!\lambda(\mc{K}_t^{(r)})\). Let \(\mf{G}^*\) be the family of all \(r\)-graphs in \(\brm{Forb}(\mc{F}^{+(t+1)})\) which cover pairs. Then \(\mf{G}^*\)  is \(\mf{B}(\mc{K}_t^{(r)})\)-weakly weight stable.
\end{lem}
\begin{proof} Let $\mf{T}$ be the family of all $r$-graphs on at most $t$ vertices not isomorphic to $\mc{K}_t^{(r)}$, and let $\lambda^* = \max\{\lambda(\mf{T}),\pi(\mc{F})/r!\}$. Then $\lambda^* < \lambda (\mc{K}_t^{(r)}).$ Thus it suffices to show that, if $\lambda(\mc{G}) > \lambda^*$ for some $\mc{G} \in \mf{G}^*$, then $\mc{G}$ is isomorphic to $\mc{K}_t^{(r)}$. 

If $\pl{v}(\mc{G}) \leq t$ then $\mc{G}$ is isomorphic to $\mc{K}_t^{(r)}$, as otherwise $\lambda(\mc{G}) \leq \lambda(\mf{T}).$ Thus we assume  $\pl{v}(\mc{G}) > t$.
Then $\mc{G}$ is $\mc{F}$-free, and, as $\mc{F}$ covers pairs, it follows that $\mf{B}(\mc{G})$ is  $\mc{F}$-free. It is not hard to see that $$\lambda(\mc{G})=\sup_{\mc{B} \in \mf{B}(\mc{G})}\frac{|\mc{B}|}{\pl{v}(\mc{B})^{r}} \leq \frac{\pi(\mc{F})}{r!} \leq \lambda^*,$$
a contradiction.
\end{proof}

\vskip 10pt
\subsection*{Acknowledgement}
When preparing the final version of this paper we have learned that Brandt, Irwin and Jiang~\cite{jiang} independently proved Theorems~\ref{thm:trees} and~\ref{thm:mubayigen}. While their proofs also mainly use stability techniques, the details are different. For example, a substantial part of~\cite{jiang} involves a symmetrization argument, paralleling~\cite{pikhurko2}. We use a generic tool, Theorem~\ref{stabilityfromlocalstability}, instead.

\bibliographystyle{amsplain}
\bibliography{lib}
\end{document}